\documentclass[reqno]{amsart}
\usepackage[T1]{fontenc}
\usepackage[utf8]{inputenc}
\usepackage[french, english]{babel}
\usepackage{amssymb, amsmath, color, textcomp, url}
\usepackage{enumerate}
\usepackage{comment}

\usepackage{xcolor}

\theoremstyle{plain}
\newtheorem{theorem}{Theorem}[section]
\newtheorem{cor}[theorem]{Corollary}
\newtheorem{prop}[theorem]{Proposition}
\newtheorem{lemma}[theorem]{Lemma}

\newtheorem{claim}[theorem]{Claim}
\newtheorem{obs}[theorem]{Observation}
\newtheorem{construction}[theorem]{Construction}

\theoremstyle{definition}
\newtheorem{remark}[theorem]{Remark}
\newtheorem{fact}[theorem]{Fact}
\newtheorem{definition}[theorem]{Definition}
\newtheorem{example}[theorem]{Example}

\newtheorem*{notation}{Notation}
\newtheorem*{ack}{Acknowledgments}

\newcommand{\nc}{\newcommand}

\nc{\Z}{\mathbb{Z}}
\nc{\N}{\mathbb{N}}
\nc{\Q}{\mathbb{Q}}
\nc{\K}{\mathbb{K}}
\nc\LL{\mathcal L}

\nc\CC{\mathcal C}

\nc{\acl}{\operatorname{acl}}
\nc{\aclq}{\operatorname{acl^{eq}}}
\nc{\dcl}{\operatorname{dcl}}
\nc{\Cb}{\operatorname{Cb}}
\nc{\lev}{\operatorname{\ell}}
\nc\inv{ ^{-1}}
\nc\Aut{\operatorname{Aut}}
\nc{\ad}{\operatorname{ad}}
\nc{\RM}{\operatorname{RM}}

\nc{\tp}{\operatorname{tp}}
\nc{\stp}{\operatorname{stp}}
\nc{\cf}{\text{cf.\xspace}}

\def\Ind#1#2{#1\setbox0=\hbox{$#1x$}\kern\wd0\hbox to 0pt{\hss$#1\mid$\hss}
\lower.9\ht0\hbox to 0pt{\hss$#1\smile$\hss}\kern\wd0}
\def\Notind#1#2{#1\setbox0=\hbox{$#1x$}\kern\wd0\hbox to
0pt{\mathchardef\nn="0236\hss$#1\nn$\kern1.4\wd0\hss}\hbox
to 0pt{\hss$#1\mid$\hss}\lower.9\ht0
\hbox to 0pt{\hss$#1\smile$\hss}\kern\wd0}

\newcommand{\oa}{\overline{a}}

\newcommand{\oc}{\overline{c}}

\newcommand{\oee}{\overline{e}}

\newcommand{\ox}{\overline{x}}
\newcommand{\oy}{\overline{y}}

\nc{\ldim}{\operatorname{ldim}}

\usepackage{tikz}
\usetikzlibrary{cd}

\newcommand{\la}{\langle}
\newcommand{\ra}{\rangle}

\newcommand{\sev}[1]{\la #1 \ra}

\newcommand{\bwsq}{\bigwedge^{\!2}\!}
\newcommand{\Bwsq}{\bigwedge^{{}_{{}_{\kern1em \scriptstyle 2}}}}

\begin{document}

\begin{abstract}
Based on Hrushovski, Palac\'{i}n and Pillay's example \cite{hrushovski2013canonical}, we produce a new structure without the canonical base property, which is interpretable in Baudisch's group. Said structure is, in particular, CM-trivial, and thus at the lowest possible level of the ample hierarchy. 
\end{abstract}

\title{CM-trivial Structures without the Canonical Base Property}
\date{\today}

\author{Thomas Blossier}
\author{Léo Jimenez}
\thanks{The second author was supported by GeoMod AAPG2019 (ANR-DFG), Geometric and Combinatorial Configurations in Model Theory}

\address{Universit\'e de Lyon; CNRS; Universit\'e Lyon 1; Institut Camille
Jordan UMR5208, 43 boulevard du 11
novembre 1918, F--69622 Villeurbanne Cedex, France}
\email{blossier@math.univ-lyon1.fr}

\address{Universit\'e de Lyon; CNRS; Universit\'e Lyon 1; Institut Camille
Jordan UMR5208, 43 boulevard du 11
novembre 1918, F--69622 Villeurbanne Cedex, France}
\email{jimenez@math.univ-lyon1.fr}
\maketitle

\tableofcontents{}

\section{Introduction}

In geometric stability, one is often interested in quantifying the complexity of forking in a given theory. An important example is one-basedness: a stable theory is one based if for any tuple $a$ and $b$, the canonical base $\Cb(\stp(a/b))$ is algebraic over $a$. This has very strong structural consequences, for example on definable groups, which must be abelian-by-finite. 

This is only the first step of a strictly increasing hierarchy of complexity: the ample hierarchy, introduced by Pillay in \cite{pillay2000note}. A theory can be $n$-ample for any $n \in \mathbb{N}$, and is $1$-ample if and only if it is not one-based. This was motivated by Hrushovski's construction of a new strongly minimal set \cite{hrushovski1993new}, which is $1$-ample, but not $2$-ample, and was the first such example. Because algebraically closed fields are $n$-ample for all $n$ \cite[Proposition 3.13]{pillay2000note} this also provided a counterexample to Zilber's trichotomy: a non one-based strongly minimal set not interpreting an algebraically closed field.

Non $2$-ampleness is also called CM-triviality in the literature, and is defined as follows: a theory $T$ is CM-trivial if whenever $A \subset B$ are parameters and $c$ is a tuple satisfying $\acl^{\mathrm{eq}}(c,A) \cap \acl^{\mathrm{eq}}(B) = \acl^{\mathrm{eq}}(A)$, then $\Cb(\stp(c/A)) \subseteq \Cb(\stp(c/B))$. 

Another way to generalize one-basedness is to introduce \emph{internality} in the definition. Recall that if $\mathcal{P}$ is a family of types, a stationary type $p \in S(A)$ is $\mathcal{P}$-internal (resp. almost $\mathcal{P}$-internal) if there is a set of parameters $C$, such that for any realization $a \models p$, there is a tuple $e$ of realizations of types in $\mathcal{P}$, each based over $C$, such that $a \in \dcl(e,C)$ (resp. $a \in \acl(e,C)$). Internality is essential to the understanding of superstable theories of finite rank, via the machinery of analysability: any type can be constructed as an iterated fibration, with $\mathcal{P}$-internal fibers at each step, where $\mathcal{P}$ is the family of Lascar rank one types. 

A relative version of one-basedness, inspired by the model theory of compact complex spaces, is the \emph{canonical base property}, which was implicitly studied in \cite{pillay2001remarks} and \cite{pillay2003jet}, and first formally defined in \cite{moosa2008canonical}. A superstable theory has the canonical base property (CBP) if (possibly working over some parameters) for any tuples $a,b$, if $\stp(a)$ has finite Lascar rank and $b = \Cb(\stp(a/b))$, then $\stp(b/a)$ is almost $\mathcal{P}$-internal, where $\mathcal{P}$ is the family of Lascar rank one types. One observes that the canonical base property is obtained by replacing "algebraic" with "almost $\mathcal{P}$-internal" in the definition of one-basedness. 

It was at first conjectured that all superstable structures of finite Lascar rank had the CBP, until Hrushovski, Palac\'{i}n and Pillay produced a counterexample \cite{hrushovski2013canonical}, which is interpretable in (and interprets) an algebraically closed field of characteristic zero. More recently, Loesch \cite{loesch2021possibly} has produced new structures without the CBP, which are conjectured to not be interpretable in the first one. 

Nevertheless, all the known examples interpret an algebraically closed field, and it is a natural extension of Zilber's trichotomy to ask if it is always the case. Moreover, the interaction between the CBP and the ample hierarchy, which are both based on generalizing one-basedness, is so far unknown. In the present article, we make progress in both of these directions by producing a CM-trivial structure that does not have the CBP. Thus, a structure without the CBP can exist at the lowest possible level of the ample hierarchy, and does not have to interpret an algebraically closed field. 

Our structure is interpretable in Baudisch's uncountably categorical group \cite{baudisch1996new}, but our methods are based on his second account of his construction \cite{baudisch2009additive} (itself inspired by methods developed in \cite{baudisch2006fusion} and \cite{baudisch2007red}). Said group  is constructed via a Hrushovski-Fraiss\'{e} amalgamation, and collapses, of finite-dimensional 2-nilpotent Lie algebras over a finite field, and was the first example of a CM-trivial superstable group. As a matter of fact, Baudisch's group is obtained from the amalgamated Lie algebra (which we will call Baudisch's Lie algebra), and we will work with the Lie algebra rather than the group. 

Our proof consists of a formal transposition of the techniques used by Hrushovski, Palac\'{i}n and Pillay in \cite{hrushovski2013canonical} to interpret a structure without the CBP in an algebraically closed field of characteristic zero. In said article, the authors carefully pick a cover of the complex numbers by their additive group to ensure that its automorphisms are given by derivations of the field. This in turn gives them enough flexibility to produce a configuration contradicting the CBP. 

Here, we will mimick their proof by first constructing derivations of Baudisch's Lie algebra. This is the technical heart of our article, and requires some elementary, but tedious, bilinear algebra. The rest of the proof is mostly routine, and a direct transposition of \cite{hrushovski2013canonical}: we pick a similar cover, and show that derivations of Baudisch's Lie algebra give rise to automorphisms of this structure. By using a criteria for the CBP first noticed in \cite{pillay2003jet}, we can copy the proof given in \cite{hrushovski2013canonical} to prove that our structure does not have the CBP.

There are still open questions regarding the canonical base property. First, the tantalizing conjecture of Hrushovski, Palac\'{i}n and Pillay that any structure interpretable in an algebraically closed field of positive characteristic has the canonical base property. At present, the authors do not see any reason to confirm or infirm this, except that the known cover constructions do not transpose to this case. Second, the existence of a structure without the CBP, and not interpreting a group. Geometric stability theory considerations show that such a structure cannot be $\aleph_1$-categorical, but there is no other known obstruction to its existence. Perhaps it could be interpretable in Hrushovski's original strongly minimal set.

The article is organized as follows: in Section 2, we give necessary preliminaries on both 2-nilpotent Lie algebras and Baudisch's construction. In Section 3, we construct many derivations on Baudisch's Lie algebra, using bilinear algebra considerations, and properties of Baudisch's construction. Finally, everything comes together in Section 4, where said derivations are used to produce a cover of Baudisch's Lie algebra without the canonical base property.

\begin{ack}

The authors would like to thank Frank Wagner for insightful discussions on the subject of this paper. The second author is also grateful to Anand Pillay for suggesting the construction of new counterexamples to the CBP as a dissertation project.

\end{ack}

\section{Preliminaries}

From now on, we denote by $\sev{A}$ the vector subspace generated by a subset $A$ in a vector space.

\subsection{$2$-nilpotent graded Lie algebras}

Before we start our journey into Baudisch's work \cite{baudisch1996new} and \cite{baudisch2009additive}, we will remind the reader of a few elementary definitions and results on 2-nilpotent Lie algebras, which will be central to our construction. 

\begin{definition}

Let $\mathbb{K}$ be a field. A $\mathbb{K}$-Lie algebra is a $\mathbb{K}$-vector space $\mathfrak{g}$ equipped with a bilinear map $[\cdot , \cdot] : \mathfrak{g} \times \mathfrak{g} \rightarrow \mathfrak{g}$, called the Lie bracket, that is:

\begin{itemize}
\item alternative, i.e $[x,x] = 0$ for all $x \in \mathfrak{g}$
\item anticommutative, i.e. $[x,y] = - [y,x]$ for all $x,y \in \mathfrak{g}$
\item satisfies the Jacobi identity, i.e. $[x,[y,z]] + [z,[x,y]] +[y,[z,x]] = 0$, for all $x,y,z \in \mathfrak{g}$.
\end{itemize}

\end{definition}

In Baudisch's construction, one considers graded $2$-nilpotent Lie algebras.

\begin{definition}

A $2$-nilpotent graded Lie algebra is a Lie algebra $\mathfrak{g}$, which is graded as a vector space, meaning $\mathfrak{g} = \mathfrak{g}_1 \oplus \mathfrak{g}_2$, and satisfies:

\begin{itemize}
\item $\la [\mathfrak{g}_1,\mathfrak{g}_1] \ra = \mathfrak{g}_2$
\item $[\mathfrak{g},\mathfrak{g}_2 ] = \{ 0 \}$
\end{itemize}
\end{definition}

A Lie algebra homomorphism is a $\mathbb{K}$-linear map preserving the Lie Bracket. Just as for commutative rings, kernels of homomorphisms will be exactly ideals:

\begin{definition}

Let $\mathfrak{g}$ be a $\mathbb{K}$-Lie algebra. A vector space $\mathfrak{h} \subseteq \mathfrak{g}$ is a subalgebra if it is preserved by the Lie bracket. It is an ideal if it moreover satisfies $[\mathfrak{g},\mathfrak{h}] \subseteq \mathfrak{h}$.

Given an ideal $\mathfrak{h} \subseteq \mathfrak{g}$, one can form the quotient Lie algebra $\mathfrak{g}/\mathfrak{h}$.

\end{definition}

\begin{remark}

Let $\mathfrak{g} = \mathfrak{g}_1 \oplus \mathfrak{g}_2$ be a $2$-nilpotent graded Lie algebra:
\begin{itemize}
\item any vector subspace $\mathfrak{g}'_1$ of $\mathfrak{g}_1$ generates a \emph{subgraded} algebra $\mathfrak{g}' =\mathfrak{g}'_1 \oplus \mathfrak{g}'_2$, i.e. a $2$-nilpotent graded subalgebra $\mathfrak{g}'$ such that $(\mathfrak{g'})_i =\mathfrak{g}_i \cap  \mathfrak{g}'$;
\item any vector subspace of $\mathfrak{g}_2$ is an ideal.

\end{itemize}

\end{remark}

In fact,  $2$-nilpotent graded Lie algebras correspond exactly to quotients of exterior squares:

\begin{definition}

Let $V$ be a vector space over a field $\mathbb{K}$. The exterior algebra $\bigwedge V$ of $V$ is defined as the quotient of the tensor algebra $T(V)$ by the (two-sided) ideal generated by $\{ x \otimes x , x \in V\}$. We denote by $x \wedge y$ the product in $\bigwedge V$, and call it the wedge product of $x$ and $y$.

The exterior square $\bwsq V$ is the vector space generated by $\{ x \wedge y, x,y \in V\}$.

\end{definition}

We will frequently state that a family of wedge products is free. When we do, it is a consequence of the following:

\begin{fact}

Let $\{ a_i, i \in I \}$ be a basis of $V$, and fix an ordering $<$ of $I$. Then $\{ a_i \wedge a_j, i<j \}$ is a basis of $\bwsq V$. 

\end{fact}

From the exterior square, we can construct 2-nilpotent Lie algebras:

\begin{obs}

The vector space $V \bigoplus \bwsq V$ can be equipped with a Lie algebra structure by setting $[x,y] = x \wedge y$ for all $x,y \in V$, and $[x,[y,z]] = 0$ for all $y,z \in V$ and $x \in V \bigoplus \bwsq V$. This Lie algebra is $2$-nilpotent graded by construction. It is the free 2-nilpotent Lie algebra over $V$, denoted $F_2(V)$.

\end{obs}

A trivial but essential remark, which is used heavily by Baudisch \cite{baudisch2009additive}, is the following:

\begin{remark}\label{2-nilpotent-as-quotient}
Consider a $2$-nilpotent graded Lie algebra $\mathfrak{g} =\mathfrak{g}_1 \oplus \mathfrak{g}_2$. The identity from $V = \mathfrak{g}_1 $ to  $\mathfrak{g}_1$ extends canonically to a Lie algebra morphism $\varphi$ from $F_2(V)$ onto $\mathfrak{g}$, which induces an Lie isomorphism between $\mathfrak{g}$ and $F_2(V)/N_\mathfrak{g}$ where $N_\mathfrak{g}$ is the kernel of $\varphi$.  
Note that 

\[N_\mathfrak{g} =\left\{\sum\limits_{i=1}^n \lambda_i x_i\wedge y_i \colon x_i,y_i \in V,\, \lambda_i \in \mathbb{K} \text{ such that } \sum\limits_{i=1}^n \lambda_i [x_i,y_i]=0\right\}.\]

Thus, one can identify $2$-nilpotent graded Lie algebras to Lie algebras of the form $V \oplus  \bwsq V\big/N$ where $V$ is a $\K$-vector space and $N$ a $\K$-vector subspace of $\bwsq V$. When there is no ambiguity, we denote $N(V)$ the considered ideal $N$ of $F_2(V)$.

Let $\mathfrak{g} = A \oplus \bwsq A\big/N(A)$ be a $2$-nilpotent graded Lie algebra and $B$ be a vector subspace of $A$. The subalgebra $\mathfrak{g}'$ generated by $B$ is isomorphic to  $B \oplus \bwsq B\big/N(B)$ where $N(B) = N(A) \cap  \bwsq B$.
\end{remark}

From linear maps on the first components of $2$-nilpotent graded Lie algebras, one can obtain Lie algebra morphisms:
\begin{remark}\label{rem_morphism} 
Let $\mathfrak{g} = U\oplus \bwsq U\big/N(U)$ and $\mathfrak{h} = V\oplus \bwsq V\big/N(V)$ be $2$-nilpotent graded Lie algebras.
\begin{itemize}
\item Any linear map $\sigma: U \rightarrow F_2(V)$ extends uniquely to a Lie algebra morphism $\sigma: F_2(U) \rightarrow F_2(V)$.
\item Moreover, if $\sigma(N(U)) \subseteq N(V)$, it induces by quotients a unique Lie algebra morphism $\sigma: \mathfrak{g} \rightarrow \mathfrak{h}$.
\end{itemize}
In the rest of the paper, we will consider only \emph{graded} Lie algebra morphisms, which are morphisms of the form $\sigma: \mathfrak{g} \rightarrow \mathfrak{h}$ such that $\sigma(U) \subseteq V$ or equivalently linear maps $\sigma: U \rightarrow V$ such that $\sigma(N(U)) \subseteq N(V)$.  
\end{remark}

Derivations play a key role in the study of Lie algebras, and will feature prominently in our construction: 

\begin{definition}

Let $\mathfrak{g}$ be a $\mathbb{K}$-Lie algebras. 

A derivation $\delta : \mathfrak{g} \rightarrow \mathfrak{g}$ is a $\mathbb{K}$-linear map satisfying the Leibniz law, that is $\delta([x,y]) = [\delta(x),y] + [x, \delta(y)]$ for all $x,y \in \mathfrak{g}$.

A partially defined derivation $\delta$ is a $\mathbb{K}$-linear map from a subalgebra $\mathfrak{g}'$ to  $\mathfrak{g}$ satisfying the Leibniz law (on elements in $\mathfrak{g}'$).

When $\mathfrak{g}$ is $2$-nilpotent graded and $\mathfrak{g}'$ is graded Lie subalgebra, we say that a partially defined derivation $\delta: \mathfrak{g}' \rightarrow \mathfrak{g}$ is \emph{graded} if $\delta( \mathfrak{g}'_1) \subseteq \mathfrak{g}_1$.

\end{definition}

Note that for any Lie algebra $\mathfrak{g}$ and $g \in \mathfrak{g}$, the application $\delta_g : x \rightarrow [x,g]$ is a derivation. However, in the case of a graded $2$-nilpotent Lie algebra, this derivation is \emph{not} graded. One can construct graded derivations just as graded morphisms were previously constructed:

\begin{remark}\label{rem-deriv}
Let $\mathfrak{g} = A \oplus \bwsq A\big/N(A)$ be a $2$-nilpotent graded Lie algebra, $B$ be a vector subspace of $A$, and $\mathfrak{g}'$ be the subalgebra generated by $B$.

\begin{itemize}
\item A linear map $f : B \rightarrow F_2(A)$ induces a unique partially defined derivation $\tilde f : F_2(B) \rightarrow F_2(A)$ such that $\tilde f_{\mid B} = f$.
\item A partially defined derivation $\delta: \mathfrak{g'} \rightarrow \mathfrak{g}$ is uniquely determined by the linear map $\delta_{\mid B} : B \rightarrow  \mathfrak{g}$. Moreover we have  $\tilde \delta_{\mid B}(N(B)) \subseteq N(A)$.
\item Reciprocally, a linear map $f : B \rightarrow F_2(A)$ such that $\tilde f (N(B)) \subseteq N(A)$ induces by quotients an unique partially defined derivation $\delta: \mathfrak{g}' \rightarrow \mathfrak{g}$.
\item If $f(B) \subseteq A$, the corresponding partially defined derivation is graded.
\end{itemize}
When there is no ambiguity, we will use the same notation for the linear map and the partially defined derivations. In such a setting, we will say that $f : B \rightarrow A$ is a partially defined derivation if $f$ is a linear map from $B$ to $A$ such that $f(N(B)) \subseteq N(A)$ (where $f$ denotes $\tilde f$ in the last inclusion).
\end{remark}

\subsection{Baudisch's group}

Our final structure will be based on Baudisch's group (or rather, the Lie algebra associated to it). In this section, entirely due to Baudisch, we recall how this algebra is constructed, and state results that we will use. This group was first constructed by Baudisch in \cite{baudisch1996new}, but we found his second article \cite{baudisch2009additive} easier to use for our purposes. We will freely refer to results and definitions from said article, and we encourage the reader to have it within reach.

Fix a finite field $\mathbb{F}_q$, with $q > 2$. 
In said article, Baudisch constructs an $\omega$-stable $2$-nilpotent graded $\mathbb{F}_q$-Lie algebra $M$ as the Hrushovski-Fraiss\'{e} limit of a class of finite $2$-nilpotent graded $\mathbb{F}_q$-Lie algebra. He then constructs a CM-trivial $2$-nilpotent graded $\mathbb{F}_q$-Lie algebra of Morley rank $2$ as a collapse $M_{\mu}$, which depends on a certain function $\mu$. 

To see 2-nilpotent graded $\mathbb{F}_q$-Lie algebras as first order structures, we will use an expansion of the language of $\mathbb{F}_q$-vector spaces by a binary function symbol $[\cdot,\cdot]$ for the Lie bracket, as well as two unary predicates for the degree one and degree two components. We denote by $L$ this language. We reserve the notation $\wedge$ for the Lie bracket in free $2$-nilpotent Lie algebras:

\begin{notation}

If $A \oplus \bigwedge^2 A /N(A)$ is a 2-nilpotent Lie algebra, we will often work in the free $2$-nilpotent Lie algebra $F_2(A)$. When we do, we will always use the wedge product notation. Thus, an equality of the form $\sum\limits_{i,j} [a_i,b_j] = 0$ in the Lie algebra $A \oplus \bigwedge^2 A /N(A)$ is equivalent to $\sum\limits_{i,j} a_i \wedge b_j \in N(A)\subseteq \bigwedge^2 A $.

We will say that $A \oplus \bigwedge^2 A /N(A)$ satisfies $\sum\limits_{i,j} a_i \wedge b_j \in N(A)$, even though this is not, strictly speaking, a formula in $L$.

\end{notation}

Let us recall how Baudisch's class of Lie algebras, as well as their predimension function, are defined. 

\begin{definition}[Predimension]

For any finite $2$-nilpotent graded Lie algebra $\mathfrak{g} = A \oplus \bwsq A \big/N(A)$, we let $$\delta(\mathfrak{g}) = \ldim(A) - \ldim(N(A))$$
where $\ldim$ denotes the $\mathbb{F}_q$-linear dimension.

When we are in an ambient $2$-nilpotent Lie algebra $V \oplus \bwsq V \big/N(V)$, for any finite vector subspace $B$ of $V$, we simply denote by $\delta(B)$ the predimension of the subalgebra generated by $B$. 

Note that $\delta(B) = \ldim(B) - \ldim(N(B))$ where $N(B) = N(V) \cap \bwsq B$.

More generally, the relative predimension is defined for vector subspaces $C \subseteq B$ of $V$ such that $B$ is finitely generated over $C$ by 
\[\delta(B/C) = \ldim(B/C) - \ldim(N(B)/N(C)).\]

\end{definition}

\begin{obs}[Sub-modularity]
The predimension is sub-modular: for two vector subspaces $A$ and $B$ such that $A$ is finitely generated over $A \cap B$, we have
\[\delta(A+B/B) \leq \delta(A/A\cap B).\]
\end{obs}

\begin{definition}[Strongness]
When working in an ambient $2$-nilpotent Lie algebra $V \oplus \bwsq V \big/N(V)$, for any vector subspaces $B \subseteq A$ of $V$, we say that $B$ is self-sufficient or strong in  $A$ (denoted by $B \leq A$), if for any vector subspace $B \subseteq C \subseteq A$, finitely generated over $B$, we have $\delta(C/B) \geq 0$. 

\end{definition}

Note that because of submodularity, the intersection of two subspaces that are strong in $V$ is also strong. Thus, given $B \subset V$, there is a smallest strong subspace of $V$ that contains $B$, namely the intersection all strong subspaces of $V$ containing $B$. We call it the self-sufficient closure of $B$.

We say that $A \supsetneq B$ is \emph{a minimal strong extension} of $B$ if $B \leq  A$ and there is no vector subspace $B \subsetneq C \subsetneq A$ such that $C \leq A$.

\begin{definition}[Class $\mathcal{K}$]
The class $\mathcal{K}$ is made of all  $2$-nilpotent Lie algebra $V \oplus \bwsq V \big/N(V)$ such that:
\begin{itemize}
\item $[v,w] \neq 0$ for any linearly independent $v,w$ in $V$;
\item $\mathbb{F}_q \cdot v \leq V$ for all $v$ in $V$.
\end{itemize}
\end{definition}

Note that the second assumption is equivalent to $\delta(A) \geq 1$ for any finite non-trivial vector subspace of $V$.

In the class $\mathcal{K}$, there are three possibilities for minimal strong extensions:

\begin{itemize}
\item Transcendental. In that case, we have $\ldim(A) = \ldim(B) + 1$ and $\delta(A) = \delta(B +1)$
\item Algebraic. In that case, we have $\ldim(A) = \ldim(B) + 1$ and $\delta(A) = \delta(B)$. This forces the existence of $a \in A \setminus B$ such that $A = B \oplus \la a \ra$ and $N(A) = N(B) \oplus \la a \wedge b + c \ra$, for $b \in B$ and $c \in \bwsq B$. 
\item Prelagebraic. This contains all the other cases, meaning that $\ldim(A) > \ldim(B) + 1$ and $\delta(A) = \delta(B)$.

\end{itemize}

\begin{example}

Let $B = \la b_0,b_1,b_2 \ra$ be a 3-dimensional $\mathbb{F}_q$-vector space, and consider the Lie graded 2-nilpotent algebra $F_2(B)$. We construct an extension of $B$ of each type, using linearly independent elements $a_0$, $a_1$ over $B$.

\begin{itemize}
\item Let $A_{\mathrm{tr}} = B \oplus \sev{a_0} $, and consider the Lie algebra $F_2(A_{\mathrm{tr}})$, this is a transcendental extension of $B$.
\item Let $A_{\mathrm{alg}} = B \oplus \sev{a_0}$, and consider the Lie algebra $F_2(A_{\mathrm{alg}})/N(A_{\mathrm{alg}})$, with $N(A_{\mathrm{alg}}) = \la a_0 \wedge b_0 + b_1 \wedge b_2 \ra$. This is an algebraic extension of $B$.
\item Let $A_{\mathrm{pr}} = B \oplus \la a_0, a_1 \ra$, and consider the Lie algebra $F_2(A_{\mathrm{pr}})/N(A_{\mathrm{pr}})$, with $N(A_{\mathrm{pr}}) = \la a_0 \wedge b_0 + a_1 \wedge b_1 , a_0 \wedge a_1 + b_1 \wedge b_2 \ra$. This is a prealgebraic extension of $B$.
\end{itemize}

\end{example}

A key property is

\begin{theorem}\cite[Theorems 8.3 and 8.4]{baudisch2009additive}

The subclass $\mathcal{K}_{fin}$ of finite $L$-structures in $\mathcal{K}$ has the amalgamation property with respect to strong embeddings. Therefore, there is a unique up to isomorphism countable $L$-structure $M = V \oplus \bwsq V /N(V)$ in $\mathcal{K}$, constructed by Hrushovski-Fraiss\'{e} amalgamation, that is rich, meaning:
if $B\leq V$ is a finite self-sufficient subspace of $V$ and $A \geq B$ is a finite strong extension of $B$ in $\mathcal{K}$, there is a strong embedding $f$ of $A$ in $V$ over $B$.
\end{theorem}

Note that because we imposed, for any $A \in \mathcal{K}$, that $[a,b] \neq 0$ for any linearly independent $a,b \in A$, any algebraic extension in $V$ is algebraic in the model-theoretic sense. However, this is not the case for prealgebraic extensions, which are of Morley rank one over their basis. This is why $V$ and $M$ have respectively Morley rank $\omega$ and $\omega \cdot 2$: we can take successive prealgebraic extensions, and reach any finite Morley rank in $V$.

We will denote $T$ the theory of the $L$-structure $M$. Note that $M$ is the countable saturated model of $T$ (\cite[Theorem 8.6]{baudisch2009additive}).

As we want a finite rank structure, what is needed is to \emph{collapse} this structure, meaning force the prealgebraic extensions to become algebraic. As is classical in Hrushovski constructions, this is done via a set of \emph{codes} for prealgebraic extensions. More precisely, Baudisch consider in \cite{baudisch2009additive} a set of \emph{good codes}, which is a set $(\varphi_\alpha(\ox, y))_{\alpha \in \mathcal{C}}$ of $L^{\mathrm{eq}}$-formula, with in particular the  following properties (and many others useful for the construction):
\begin{itemize}
\item $\ox$ is a $n$-tuple of variables in the first predicate of $L$ (that is, a tuple of elements of degree 1) and $y$ is an imaginary variable,
\item for $b \in M^{eq}$, either the formula $\varphi_\alpha(V^n,b)$ is empty or $b$ is the canonical parameter of $\varphi_\alpha(\ox,b)$ in $T$,
\item there exist $n$ terms 
\[\Phi_i(\ox,\oy_i,z_i) = \sum_{j<k<n} \lambda_{ijk} x_j \wedge x_k + \sum_{j<n} y_{ij} \wedge x_j + z_i\]
where $\oy_i$ is a tuple of variables in the first predicate and $z_i$ a variable of the second predicate, which describe up to isomorphims a prealgebraic extension in the following sense:

if $b \in \dcl^{eq}(B)$ for a vector subspace $B$ of $V$ (and $\varphi_\alpha(V^n,b)\neq \emptyset$) there exist $c_{ij} \in B$ and $\psi_i \in \bwsq B$ such that
for any $\oa \in \varphi_\alpha(V^n,b)$:
\begin{itemize}
\item $\Phi_i(\oa,\oc_i,\psi_i) \in N(V)$,
\item if $\oa$ is \emph{$w$-generic} over $B$, that is $\oa$ is linearly independent over $B$ and $\delta(\sev{\oa,B}/B)=0$, then $\sev{\oa,B}$ is a prealgebraic minimal extension of $B$ where
\[N(\sev{\oa,B}) = N(B) \oplus \sev{\Phi_i(\oa,\oc_i,\psi_i) \colon i<n}.\]
\end{itemize}
\end{itemize}

Then using a notion of \emph{difference sequences} for good codes, a subclass $\mathcal{K}^\mu$ of $\mathcal{K}$ is defined for any \emph{good map} $\mu: \mathcal{C} \rightarrow \mathbb{N}$, such that:
\begin{theorem}[\cite{baudisch2009additive,baudisch1996new}]
The subclass $\mathcal{K}^{\mu}_{fin}$ of finite $L$-structures in $\mathcal{K}^{\mu}$ has the amalgamation property with respect to strong embeddings and the theory $T^\mu$ of the countable rich structure $M_\mu= V_\mu \oplus \bwsq V_\mu\big/ N(V_\mu)$ of $\mathcal{K}^{\mu}$ is uncountably categorical of Morley rank $2$, with $V_\mu$ being strongly minimal. Moreover $T^\mu$ is $CM$-trivial.  
\end{theorem}

The \emph{Bounds for Difference Sequences} \cite[Section 5]{baudisch2009additive} is central in the above construction. In particular there is a characterization of minimal extensions which do not belong in $\mathcal{K}^\mu$ \cite[Corollary 5.3]{baudisch2009additive}. In order to extend derivations, we will need only the following fact that one can deduce directly from this characterization:

\begin{fact}\label{fact_out_Kmu}
Let $D$ be in $\mathcal{K}_{\mu}$ and $D \leq D'$ be a minimal extension in $\mathcal{K} \setminus \mathcal{K}_{\mu}$. Then the extension $D'$ is prealgebraic, and there is a good code $\alpha \in \mathcal{C}$ such that one of the following holds:

\begin{enumerate}[(a)]
\item $D'= D+ \sev{\oee}$ for a $w$-generic realization $\oee$ over $D$ of $\varphi_{\alpha}(\oee,b)$ where $b \in \dcl^{eq}(D)$, and there is at least one realization $\oee'$ in $D$ of $\varphi_{\alpha}(\ox,b)$;
\item There is a vector subspace $D \subseteq E \subseteq D'$ and $b \in \dcl^{eq}(E)$, with at least two realizations $\oee$ and $\oee'$ in $D'$ of $\varphi_{\alpha}(\ox,b)$ where $\oee$ is $w$-generic over $E$ and $\oee'$ is $w$-generic over $E+\sev{\oee}$.
\end{enumerate}
\end{fact}

\section{Constructing Derivations}

To construct derivations on $M_\mu$, we will proceed by induction, namely extending derivations step-by-step. We will consider only partially defined graded derivations on finite $2$-nilpotent Lie algebras, or equivalently in an ambient $2$-nilpotent Lie algebras $V \oplus \bwsq V \big/N(V)$, we consider partially defined derivations $f : B \rightarrow V$ where  $B$ is a finite vector subspace of $V$ and $f : B \rightarrow V$ a linear map such that $f(N(B)) \subseteq N(B+f(B))$ (see Remark \ref{rem-deriv}).

\begin{definition}

A \emph{derivation extension problem} is the data of two finite vector subspaces $B \leq A \leq V_{\mu}$ and a partially defined graded derivation $f : B \rightarrow  V_\mu $, with  $A+f(B) \leq V_{\mu}$. 

We denote it by $(B \leq A, f)$

\end{definition}

For any derivation extension problem $(B \leq A,f)$, we want to extend the derivation to $A$, and obtain a partially defined graded derivation $f : A \rightarrow V_\mu$, with $A+f(A)\leq V_{\mu}$.

Let us fix such a derivation extension problem $(B \leq A, f)$ for now. 

First we construct what we will call a free pseudosolution. The Lie algebra generated by $A+f(A)$ for this free pseudosolution will not be in general in $\mathcal{K}_\mu$, and finding a solution realized in $M_\mu$ will be the technical heart of the proof.

\begin{construction}
Consider an \emph{abstract} vector space $U$ over $A+f(B)$ such that $\ldim(U/A+f(B)) =\ldim(A/B)$. Extend the linear map $f: B \rightarrow B+f(B)$ to a linear map $f: A \rightarrow U$ which sends a (any) basis of $A$ over $B$ onto a basis of $U$ over $A+f(B)$. From now, we set $U = A+f(A)$.

By Remark \ref{rem-deriv}, this map canonically gives us a partially defined derivation $f: F_2(A) \rightarrow F_2(A+f(A))$.

Now we define $N(A+f(A)) = N(A+f(B))+ f(N(A))$ and we consider the $2$-nilpotent graded algebra 
$\mathfrak{g} = (A+f(A)) \oplus \bwsq \big(A+f(A)\big)/N(A+f(A))$. Then we have to check that $N(A+f(B)) = N(A+f(A)) \cap \bwsq (A+f(B))$ in order to prove that the Lie subalgebra generated by $A+f(B)$ in $M_\mu$ is also a subgraded algebra of $\mathfrak{g}$. We will use a particular (and simpler) case of the following lemma for $A'=B$.
\end{construction}

\begin{lemma}\label{indep-wedges} For any $B \subseteq A' \subseteq A$, we have $f^{-1}\left(\bwsq \big(A + f(A')\big)\right) = \bwsq A'$.

Thus, if the family $(e_1, \ldots, e_t)$ of vectors in $\bwsq A$ is free over $\bwsq A'$, then the family  $(f(e_1), \ldots, f(e_t))$ is free over  $\bwsq \big(A + f(A')\big)$.
\end{lemma}

\begin{proof}
Consider a basis $(a_1, \ldots, a_n)$ of $A$ over $B$ such that $A' = \sev{ B, a_1, \cdots , a_s}$. 

Let $e \in \bwsq A$, we can write: 

\begin{align*}
e = \sum\limits_{i,j} \lambda_{i,j}(a_i \wedge b_j) + \sum\limits_{k < \ell} \beta_{k,\ell} (a_k \wedge a_\ell) + c 
\end{align*}
with $\lambda_{i,j}$, $\beta_{k,\ell} \in \mathbb{F}_q$, $c \in \bwsq B$ and $b_1,\ldots,b_m$ linearly independent vectors of $B$. 
Then
\begin{align*}
f(e) & = \sum\limits_{i,j} \lambda_{i,j}(f(a_i) \wedge b_j + a_i \wedge f(b_j)) + \sum\limits_{k < \ell} \beta_{k,\ell} (f(a_k) \wedge a_\ell  + a_k \wedge f(a_\ell)) + f(c) \\
& = \sum\limits_{i,j} \lambda_{i,j}(f(a_i) \wedge b_j)   -  \sum\limits_{\ell <k} \beta_{\ell,k} (a_k \wedge f(a_\ell)) + \sum\limits_{k<\ell} \beta_{k,\ell} (a_k \wedge  f(a_\ell)) \\
& + \sum\limits_{i,j} \lambda_{i,j}(a_i \wedge f(b_j)) + f(c).
\end{align*}

The last two terms of that sum belong to $\bwsq \big(A + f(B)\big)$. Moreover, the family $\{ f(a_i) \wedge b_j \colon  s < i \leq n, 1 \leq j \leq m\} \cup \{a_k \wedge f(a_\ell) \colon 1 \leq k \leq n ,  s< \ell \leq n \}$ is linearly independent over $\bwsq \big(A + f(A')\big) = \bwsq \big(A + f(B) \oplus \la f(a_1) , \cdots f(a_s) \ra\big)$. Thus if $f(e) \in \bwsq\big(A + f(A')\big)$, we must have $\lambda_{i,j} = 0$ for all $i> s$ and $\beta_{k,\ell} = 0$ for all $\ell > s$, implying $e  \in \bigwedge^2 A'$. 

This yields $f^{-1}\left(\bwsq\big(A + f(A')\big)\right) \subseteq \bwsq A'$, and the other inclusion is immediate.

\end{proof}

\begin{cor} $N(A+f(B)) = N(A+f(A)) \cap \bwsq (A+f(B))$. 

Thus, the Lie subalgebra generated by $A+f(B)$ in $M_\mu$ is also a subgraded algebra of $\mathfrak{g}= (A+f(A)) \oplus \bwsq \big(A+f(A)\big)/N\big(A+f(A)\big)$ and, the linear map $f: A \rightarrow A+f(A)$ induces a partially defined graded derivation $f: A \oplus \bwsq A \big/ N(A) \to \mathfrak{g}$.

\end{cor}
\begin{proof}
Recall that $N(A+f(A)) = N(A+f(B))+f(N(A))$. Let  $e=e_1+f(e_2)$ with $e \in \bwsq \big(A+f(B)\big)$, $e_1 \in N(A+f(B))$ and $e_2 \in N(A)$. Then $f(e_2) \in \bwsq \big(A+f(B)\big)$, and by the previous lemma $e_2 \in \bwsq B$. Since $N(B) = B \cap N(A)$ and $f:B \rightarrow B+f(B)$ is a derivation, we obtain that $f(e_2) \in N(B+f(B))$ and $e \in  N(A+f(B))$.

By definition of $N(A+f(A))$, we have $f(N(A)) \subseteq N(A+f(A))$, then, by Remark \ref{rem-deriv}, the linear map $f: A \rightarrow A+f(A)$ induces a partially graded defined derivation $f: A \oplus \bwsq A \big/ N(A) \to \mathfrak{g}$, which extends the  partially graded defined derivation $f: B \oplus \bwsq B \big/ N(B) \to (B+f(B)) \oplus \bwsq \big(B+f(B)\big) \big/ N\big(B+f(B)\big)$.
\end{proof}

We call the above extension the \emph{free pseudosolution} of the derivation extension problem (it is unique up to isomorphisms). When there is no ambiguity, we will say $A+f(A)$ the free pseudosolution of $(B \leq A,f)$.

The following lemma will be useful in understanding this derivation:

\begin{lemma}\label{dimension-equality} 
Let $A+f(A)$ be the free pseudosolution of a derivation extension problem $(B \leq A, f)$ and $A'$ a vector subspace of $A$ containing $B$.

Then $\delta(A+f(A')/A+f(B))= \delta(A'/B)$.

More precisely, the family $(f(e_1),\ldots,f(e_t))$ is a basis of $N(A+f(A'))$ over $N(A+f(B))$ for any basis $(e_1,\ldots,e_t)$ of $N(A')$ over $N(B)$.
\end{lemma}
\begin{proof}

By construction, we know that $N(A + f(A')) \subset N(A + f(B))  +f(N(A))$. As in the previous proof, let $\eta = \eta_1 + f(\eta_2)$ be in $N(A + f(A'))$, decomposed along this direct sum, i.e $\eta_1 \in N(A+f(B))$, $\eta_2 \in N(A)$ and $f(\eta_2) \in N(A+f(A'))$. By Lemma \ref{indep-wedges} we obtain $\eta_2 \in N(A) \bigcap \bwsq A'=N(A')$. Thus $N(A+f(A')) = N(A+f(B)) + f(N(A'))$.

Consider a basis $(e_1,\ldots, e_t)$ of $N(A')$ over $N(B)$. Since $N(B) = N(A') \cap \bwsq B$, the family $(e_1,\ldots, e_t)$ is free over $\bwsq B$, and, by Lemma \ref{indep-wedges}, the family $(f(e_1),\ldots, f(e_t))$ is free over $\bwsq \big(A+f(B)\big)$. Moreover, $f(N(A'))$ is generated by  $(f(e_1),\ldots, f(e_t))$  over $f(N(B)) \subseteq N(B+f(B))$ (since $f: B \rightarrow B+f(B)$ is a derivation). Thus $(f(e_1),\ldots, f(e_t))$ is a basis of  $N(A+f(B)) + f(N(A'))=N(A+f(A'))$ over $N(A+f(B))$, and $\ldim(N(A+f(A'))/N(A+f(B))= \ldim(N(A')/N(B))$.

Finally, by construction $\ldim(A+f(A')/A+f(B))=\ldim (A'/B)$, so $\delta(A+f(A'))= \delta(A'/B)$.

\end{proof}

We will consider \emph{minimal} derivation extension problems $(B \leq A, f)$, i.e. derivation extension problems such that $A$ is a minimal strong extension of $B$.

\begin{cor}\label{minimal-free-pseudosolution}

The free pseudosolution $A+f(A)$ of a (minimal) derivation extension problem $(B \leq A, f)$ is a (minimal) strong extension of $A+f(B)$.

\end{cor}

\begin{proof}
Consider a vector subspace $C$, such that  $A + f(B) \subseteq C \subseteq A + f(A)$. By linearity, $C = A+f(A')$ where $A'=f^{-1}(C)$. Note that $B \subseteq A' \subseteq A$. By the previous lemma, $\delta(C/A+f(B)) = \delta(A'/B) \geq 0$ since $B \leq A$. One deduces that $A+f(A)$ is a strong extension of $A+f(B)$. 

Suppose now that $A$ is a minimal strong extension of $B$. If $\ldim(A/B) =1$, then $\ldim(A+f(A)/B+f(A))=1$ and $A+f(A)$ is minimal over $A+f(B)$. Otherwise, $A$ is prealgebraic over $B$ and, in particular, $\delta(A/B)=0$. In this case, if  $A + f(B) \subsetneq C \subsetneq A + f(A)$, we have $\delta(C/A+f(B))= \delta(A'/B)>0 = \delta(A/B)=\delta(A+f(A)/A+f(B))$, therefore $C$ is not strong in $A+f(A)$.
\end{proof}

Our goal is to use the Lie algebra just constructed to extend a derivation from $B \leq V_{\mu}$ to $B \leq A \leq V_{\mu}$. It is not guaranteed that the free pseudosolution will belong to $\mathcal{K}_{\mu}$. There are multiple cases to consider:

\begin{enumerate}[(A)]
\item The free pseudosolution does not belong to $\mathcal{K}$. 
\item The free pseudosolution belongs to $\mathcal{K}$, but not to $\mathcal{K}_{\mu}$.
\item The free pseudosolution belongs to $\mathcal{K}_{\mu}$. In that case, we will be able to extend the derivation without further work, by using richness of $M_\mu$. 
\end{enumerate}

Let us start dealing with case (A). We first notice:

\begin{lemma}
Let $A+f(A)$ be the free pseudosolution of a derivation extension problem $(B \leq A, f)$.
For any non-zero vector subspace $C$ of $A + f(A)$, we have $\delta(C) \geq 1$.

\end{lemma}

\begin{proof}
Let $C \subseteq A + f(A)$ be a non-zero vector subspace. 

If $C \cap (A+f(B))$ is non-zero then, by submodularity, 
\[\delta(C) \geq \delta(C\cap (A+f(B))) + \delta(C+A+f(B)/A+f(B)) \geq 1+0 \]
since $A+f(B) \in \mathcal{K}$ and $A+f(B) \leq A+f(A)$.

Otherwise, consider a basis $\big(f(a_1)+a'_1,\ldots, f(a_s)+a'_s\big)$ of $C$. Since we have $C \cap (A+f(B)) = \{0\}$, the family $\big(f(a_1),\ldots,f(a_s)\big)$ is free over $A+f(B)$, and 
the basis $\big((f(a_i)+a'_i)\wedge (f(a_j)+a'_j) \colon i<j)\big)$ of $\bwsq C\subseteq \bwsq \big(A+f(A)\big)$ is free over $\big\la\alpha \wedge \alpha', f(\beta) \wedge f(\beta'), \alpha \wedge f(\alpha') \colon \alpha,\alpha' \in A,\, \beta,\beta' \in B \big\ra$.

By construction \[N(A+f(A)) \subseteq \big\la\alpha \wedge \alpha', f(\beta) \wedge f(\beta'), \alpha \wedge f(\alpha') \colon \alpha,\alpha' \in A,\, \beta,\beta' \in B \big\ra\]
and thus $N(C) = N(A+f(A)) \bigcap \bwsq C = \{0\}$. Then $\delta(C) = s >0$.

\end{proof}

Therefore, the only way for the free pseudosolution of $(B \leq A,f)$ to not belong to $\mathcal{K}$ is for $F_2(A+f(A))$ to contain linearly independent vectors $v_0,v_1$ such that $v_0 \wedge v_1 \in N(A + f(A))$, or equivalently $[v_0,v_1] = 0$ in $A+f(A) \oplus \bigwedge^2 (A+f(A))/N(A+f(A))$. This possibility cannot be eliminated in general, as the following shows:

\begin{example}

Suppose that we have three vectors $b_0,b_1,b_2$ such that $f(b_1) = f(b_2) = 0$ and $f(b_0) = b_0$, and want to extend $f$ to an element $a$ satisfying $[a, b_0] + [b_1 , b_2] = 0$.

Let $A + f(A)$ be the free pseudosolution of this derivation extension problem. It has to satisfy $[f(a) , b_0] + [a , b_0] = 0$, which can be factorized as $[f(a) + a , b_0] = 0$. Therefore the free pseudosolution cannot belong to $\mathcal{K}$.

An obvious workaround, in that case, is to consider the linear map $g$ with $g(b_i) = f(b_i)$ for $i = 0,1,2$, and $g(a) = -a$, and extend it into a derivation of $\la a,b_0,b_1,b_2 \ra$. This is easily checked to quotient into a derivation extending $f$.

\end{example}

The solution presented in the previous example is the idea behind the general case:

\begin{lemma}[Case (A)]\label{free-pseudosolution-not-inK}
Suppose that the free pseudosolution of a minimal derivation extension problem $(B \leq A, f)$ is not in $\mathcal{K}$ Then the linear map $f:B \rightarrow B+f(B)$ can be extended to a linear map $g : A \rightarrow A + f(B)$ such that $g(N(A)) \subseteq N(A+f(B))$, which gives a solution of the extension problem.

\end{lemma}

\begin{proof} By the previous lemma, in such a case, the free pseudosolution $A+f(A)$ contains linearly independent vectors $v_0,v_1$ such that $v_0 \wedge v_1 \in N(A+f(A))$ (i.e. such that $[v_0,v_1]=0$).

As $N(A + f(A))$ is generated by $f(N(A))$ over $N(A + f(B))$, there are $e \in N(A)$ and $c \in N(A + f(B))$ such that $f(e) + c = v_0 \wedge v_1$.

Note first that $f(e)+c \in \big \la f(a) \wedge a' \colon a, a' \in A \big\ra + \bwsq \big( A +f(B)\big)$.

We can write $v_0 = f(a_0)+a'_0$ and $v_1 = f(\tilde a_1)+ \tilde a'_1$ with $a_0, a'_0,\tilde a_1, \tilde a'_1 \in A$. Since $M_\mu \in \mathcal{K}$, at least one of $v_i \notin A+f(B)$. We may assume that $a_0 \in A \setminus B$.

Remark that $f(\tilde a_1) \in \la A,f(a_0)\ra$: otherwise, 
\[f(a_0) \wedge f(\tilde a_1) \notin \big \la f(a) \wedge a' \colon a, a' \in A \big\ra + \Bwsq \big( A +f(B)\big),\] which contradicts the equality 
\[ v_0 \wedge v_1 = f(a_0) \wedge f(\tilde a_1) +f(a_0) \wedge \tilde a'_1  + a'_0 \wedge  f(\tilde a_1) +a'_0 \wedge \tilde a'_1 =f(e)+c.\]

Thus, there is $\alpha \in \mathbb{F}_q$ and $a'_1 \in A$, such that $v_1 = \alpha f(a_0)+a'_1$, and in fact
\[ v_0 \wedge v_1 = f(a_0) \wedge (a'_1  -\alpha a'_0) +a'_0 \wedge  a'_1.\]

Now, complete $a_0$ to a basis $(a_0,\ldots,a_n)$ of $A$ over $B$. We can write 
\[ e = \sum \limits_{i < j \leq n} \lambda_{i,j} a_i \wedge a_j + \sum\limits_{i \leq n} a_i\wedge b_i + d\]
with $\lambda_{i,j} \in \mathbb{F}_q$, $b_i \in B$ and $d \in \bwsq B$.

Then
\[f(e) = \sum \limits_{i < j \leq n} \lambda_{i,j} f(a_i) \wedge a_j -  \sum \limits_{j<i \leq n} \lambda_{j,i} f(a_i) \wedge a_j +c'\] with 
$c' \in \big \la f(a) \wedge b \colon a \in A,\, b\in B \big\ra + \bwsq \big( A +f(B)\big)$.

But the family $\big\{f(a_i)\wedge a_j \colon i\neq 0,\, j \neq i\big\}$ is free over 
\[\big \la f(a_0) \wedge a' \colon a' \in A \big\ra+  \big \la f(a) \wedge b \colon a \in A,\, b\in B \big\ra + \Bwsq \big( A +f(B)\big).\]
Therefore the equality $f(e)+c = v_0\wedge v_1$ yields that $\lambda_{j,i} =0$ for all $j<i\leq n$ and thus 
\[e = \sum\limits_{i \leq n} a_i\wedge b_i + d.\]

This yields \[f(e) = \sum\limits_{i \leq n} f(a_i)\wedge b_i +c^*\]
with $b_i \in B$ and $c^* \in \bwsq \big( A +f(B)\big)$.

Note that if $(a_0^*,\ldots,a_n^*)$ is a family of vectors of $A$, then the family $\{f(a_i)\wedge a_i^*  \colon a^*_i \neq 0, i = 0 , \cdots , n \}$ is free over $\bwsq \big( A +f(B)\big)$. Thus, 
the equality $f(e)+c = v_0\wedge v_1$ yields that $b_0 = a'_1-\alpha a'_0$ and $b_i =0$ for all $i>0$.

We have  $e = a_0 \wedge b_0 +d$, where $a_0 \in A \setminus B$, $b_0 \in B$ and $d \in \bwsq B$. Since $v_0$, $v_1$ are linearly independent we have $b_0 = a'_1-\alpha a'_0 \neq 0$. Thus, $e \in N(A) \setminus N(B)$ and $B \oplus \sev{a_0}$ defines an algebraic extension of $B$. By minimality of $B \leq A$, we deduce that $A = \sev{B,a_0}$.

Since $a'_1 = \alpha a'_0 + b_0$, we have $v_0 \wedge v_1 = f(a_0) \wedge b_0 + a'_0 \wedge a'_1 = f(a_0) \wedge b_0 + a'_0 \wedge b_0$.
Then the equality $f(e)+c = v_0 \wedge v_1$ yields the following equality in $\bwsq \big(A+f(B)\big)$  
\[a_0 \wedge f(b_0) +f(d)+c = a'_0 \wedge b_0.\]

Let us consider the linear map $g : A \rightarrow A+f(B)$ such that $g \vert _{B} = f$ and $g(a_0) = -a'_0$. We claim that $g$ is a solution to the derivation extension problem $(B \leq A, f)$. 

By Remark \ref{rem-deriv}, it is enough to show that $g(e) \in N(A+f(B))$. We have:
\begin{align*}
g(e) & = g\left(a_0 \wedge b_0+ d\right) \\
& = g(a_0) \wedge b_0 + a_0 \wedge g(b_0) + g(d) \\
& = -a'_0 \wedge b_0 + a_0 \wedge f(b_0) + f(d) \\
& = -c \in N(A+f(B)).
\end{align*}

This $g$ is thus the solution we were looking for.

\end{proof}

Now consider the case (B):
\begin{lemma}[Case (B)]\label{free-pseudosolution-not-inKmu}
Suppose that the free pseudosolution of a minimal derivation extension problem $(B \leq A, f)$ belongs to $\mathcal{K}$, but does not belong to $\mathcal{K}_{\mu}$. Then the linear map $f:B \rightarrow B+f(B)$ can be extended again to a linear map $g : A \rightarrow A + f(B)$ such that $g(N(A)) \subseteq N(A+f(B))$, which gives a solution of the extension problem.
\end{lemma}

\begin{proof}

By Corollary \ref{minimal-free-pseudosolution} $A+f(A)$ is a minimal strong extension of $A+f(B)$ and thus, by Fact \ref{fact_out_Kmu}, this can happen for two different reasons:

\begin{enumerate}[(a)]
\item There is a good code $\alpha \in \mathcal{C}$ and $b \in \dcl^{eq}(A+f(B))$ such that $A+f(A)= \langle A+f(B), \oee \rangle$ with $\oee$ w-generic in $\varphi_\alpha(\ox,b)$ over $A+f(B)$. Moreover, there is $\oee' \in A+f(B)$ which realizes  $\varphi_\alpha(\ox,b)$ (in this case $\ox$ is an $n$-tuple where $n = \ldim(A/B)$);
\item There exist a vector subspace $A+f(B) \subseteq E \subseteq A+f(A)$, a good code $\alpha \in \mathcal{C}$, an imaginary $b \in \dcl^{eq}(E)$, and realizations $\oee$ and $\oee'$ of $\varphi_\alpha(\ox,b)$ such that $\oee$ is w-generic over $E$ and $\oee'$ is w-generic over $\sev{E,\oee}$ (in this case $\ox$ is an $m$-tuple with $m<\ldim(A/B)$).
\end{enumerate}

Let us take care of subcase (a) first. In this case, there are $n$ terms $(\Phi_i)_{i<n}$:  
\[\Phi_i(x_0,\ldots,x_{n-1}) = \sum_{j<k<n} \lambda_{ijk} x_j \wedge x_k + \sum_{j<n} c_{ij} \wedge x_j + \psi_i\]
where $\lambda_{ijk} \in \mathbb{F}_p$, $c_{ij} \in A+f(B)$, $\psi_i \in \bwsq \big( A +f(B)\big)$, and such that 
$\Phi_i(\oee') \in N(A+f(B))$ for $i<n$ and
\[N(A+f(A)) = N(A+f(B)) \oplus \langle \Phi_i(\oee) \rangle_{i<n}.\]

Consider $\pi$  the linear map $A + f(A) \rightarrow A + f(B)$ such that $\pi\vert_{A + f(B)} = \mathrm{Id}$ and $\pi(\oee) = \oee'$. 

\begin{claim}

The linear map $\pi$ induces a Lie algebra morphism from $A+f(A)$ onto $A + f(B)$.

\end{claim}

\begin{proof}
Remember that $\pi$ extends canonically to a Lie algebra morphism from $F_2(A+f(A))$ to $F_2(A+f(B))$. Then, $\pi(\Phi_i(\oee)) = \Phi_i(\oee')$ and thus $\pi(N(A+f(A)) = N(A+f(B))$. By Remark \ref{rem_morphism}, $\pi$ induces a Lie algebra morphism from $A+f(A)$ to $A+f(B)$.
\end{proof}

Now, consider the linear map $g = \pi \circ f : A \rightarrow A+f(B)$.
\begin{claim} The linear map $g$ induces a partially defined graded derivation from $A$ to $A+f(B)$.
\end{claim}

\begin{proof}
By Remark \ref{rem-deriv} we have to check that $g(N(A)) \subseteq N(A+f(B))$ where $g$ is the canonically partially defined derivation from $F_2(A)$ to $F_2(A+f(B))$. Note that $g$ on $F_2(A)$ is equal to the composition of the Lie algebra morphism $\pi$ on $F_2(A+f(A))$ with the derivation $f$ on $F_2(A)$. Indeed, if $x,y \in A$, we obtain
\begin{align*}
\pi(f(x\wedge y)) & = \pi(f(x) \wedge y + x \wedge f(y))= \pi(f(x))\wedge \pi(y)+ \pi(x) \wedge \pi(f(y)) \\
& = \pi(f(x)) \wedge y + x \wedge \pi(f(y)) = g(x) \wedge y + x \wedge g(y) = g(x \wedge y).
\end{align*}

Thus, \[g(N(A)) = \pi(f(N(A))\subseteq \pi(N(A+f(A))= N(A+f(B)).\]

\end{proof}

Since the partially defined derivation $g : A  \rightarrow A+f(B)$ extends the derivation $f_{\vert B} : B \rightarrow B+f(B)$, the previous claim gives a solution in  $\mathcal{K}_{\mu}$ to our derivation extension problem.

We are now going to show that in the specific case of derivation extension problems, subcase (b) cannot happen.

By way of contradiction, suppose that $A + f(B) \leq A + f(A)$ is of type (b). In this case,
\[\ldim(E+\langle \oee,\oee'\rangle/E) = \ldim(N(E+\langle \oee,\oee'\rangle)/N(E))=2m\] 
and the linear map which sends $\oee$ on $\oee'$ over $E$ induces a Lie algebra isomorphism over $E$ between $E+\langle \oee \rangle$ and $E+\langle \oee' \rangle$.

Thus, there are $m$ terms $(\Phi_i)_{i<m}$:  
\[\Phi_i(x_0,\ldots,x_{m-1}) = \sum_{j<k<m} \lambda_{ijk} x_j \wedge x_k + \sum_{j<m} c_{ij} \wedge x_j + \psi_i\]
where $\lambda_{ijk} \in \mathbb{F}_p$, $c_{ij} \in E$, $\psi_i \in \bwsq E$, and such that 
\[N(E+\langle \oee,\oee'\rangle) = N(E) \oplus \langle \Phi_i(\oee), \Phi_i(\oee') \rangle_{i<n} \subseteq N(A+f(A)).\]

Recall that 
\[N(A+f(A)) =N(A+f(B))+f(N(A)) \subseteq  \Bwsq \big( A +f(B)\big) + \langle f(a)\wedge a' \colon a,a'\in A\rangle.\] 

Since $A+f(B) \subseteq E \subseteq A+f(A)$ and $\oee$ is a tuple of $m$ linearly independent vectors in $A+f(A)$ over $E$, there exist a vector subspace $B \subseteq A_0 \subset A$ such that $E = A+f(A_0)$ and linearly independent vectors $a_0,\ldots,a_{m-1} \in A$ over $A_0$ such that $\oee = (f(a_0)+v_0,\ldots, f(a_{m-1})+v_{m-1})$ where $v_j \in A+f(B)$. 
If $(c_l)_{l<m}$ is a family of vectors in $E \setminus (A+f(B))$, the family 
\[(f(a_j)\wedge f(a_k), c_l\wedge f(a_l) \colon j<k<m, l<m)\] is free over $\bwsq E + \langle f(a)\wedge v \colon a \in A,v\in A+f(B)\rangle$. It follows that $\lambda_{ij} =0$ and $c_{ij} \in A+f(B)$ for all $i,j$.

The difference between the equations $\Phi_i(\oee)$ and $\Phi_i(\oee')$ gives us $m$ linearly independent equations in $N(A+f(A))$ over $N(A+f(B))$: 

\begin{align*}
\sum_{j<m} c_{ij}\wedge (e_j-e'_j) \in N(A+f(A)).
\end{align*}

Then, $\delta(A+f(B)+\langle e_j-e'_j \colon j<m\rangle /A+f(B)) = 0$ which contradicts the minimality of the extension $A+f(A)$ over $A+f(B)$ (since $m<n = \ldim(A+f(A))/A+f(B))$).

Therefore subcase (b) cannot happen.
\end{proof}

Finally, in case (C), we conclude by using richness of $M_\mu$:

\begin{prop}\label{der-ext-sol}
Every  minimal derivation extension problem $(B \leq A,f)$, has a solution: i.e. there is a partially defined derivation $g : A \rightarrow A+g(A)$ extending $f$, such that $A + g(A) \leq V_{\mu}$.
\end{prop}

\begin{proof}

For any such extension problem, if the free pseudosolution $A+f(A)$ is in $\mathcal{K}_\mu$, case (C), then by richness of $M_{\mu}$, there is a strong embedding $h$ of $A+f(A)$ in $V_{\mu}$ over $A + f(B)$ and we can take $g= h\circ f$. Otherwise, we can extend $f$ by a partially defined derivation $g : A \rightarrow A+f(B)$ (Lemma \ref{free-pseudosolution-not-inK} and \ref{free-pseudosolution-not-inKmu}).

\end{proof}

Let us show how we can use this to construct derivations on $M_{\mu}$ by reducing every configuration to a minimal derivation extension problem. Given a partially defined derivation $f$ on $B$ and a strong extension $A$ of $B$, there is no reason \emph{a priori} for $A+f(B)$ to be strong in $V_\mu$ and we cannot directly apply the previous proposition in order to extend the derivation. In that case, we will extend in several steps to the self-sufficient closure of $A+f(B)$: 

\begin{lemma}\label{der-sol-min}
Let $B \leq V_{\mu}$ be a finite strong subspace and $f : B \rightarrow B + f(B)$ a partially defined graded derivation with $B + f(B) \leq V_{\mu}$. 

If $B \leq A \leq V_\mu $ is a minimal strong extension of $B$ in $V_{\mu}$ such that $A+f(B)$ is not self-sufficient in $V_{\mu}$, then 
\begin{itemize}
\item $\delta(A/B)=1$ and
\item the partially defined derivation $f$ can be extended to a partially defined graded derivation on $\tilde{A}$ where $\tilde{A}$ is the self sufficient closure of $A+f(B)$ in $V_\mu$.
\end{itemize}

\end{lemma}

\begin{proof}

Since $B \leq A \leq V_\mu$ and $B+f(B) \leq V_\mu $, by submodularity of $\delta$, we have \[B \leq A \cap (B + f(B)) \leq A.\] 
By minimality of $B \leq A$, there are two cases, either $A \cap (B + f(B)) = A$ or $A \cap (B + f(B)) = B$.
In the first case,  $A \subseteq B + f(B)$, and since $B \subseteq A$, we get $A + f(B) = B + f(B) \leq V_{\mu}$, a contradiction. 

Thus,  $A \cap (B + f(B)) = B$. Again, by submodularity, \[0\leq \delta(A + f(B)/B + f(B)) \leq \delta(A/B) \leq 1.\]Since $A + f(B)$ is not self-sufficient in $V_{\mu}$, we have necessarily \[\delta(A + f(B)/B + f(B)) = \delta(A/B)=1.\]

As $B \leq B + f(B)$, we decompose it into a tower of minimal extensions $B = B_0 \leq \cdots \leq B_n = B + f(B)$. Note that $B_1 \cap (B + f(B)) = B + f(B)$ as $B_1 \subset B + f(B)$. By the previous discussion, this yields $B_1 + f(B) \leq V_{\mu}$, and we can extend $f$ to $B_1$ using Proposition \ref{der-ext-sol}. Iteratively, we extend $f$ to $B + f(B)$.
Let $\tilde{A}$ be the self sufficient closure of $A+f(B)$. Because $A + f(B)$ is not self-sufficient in $V_{\mu}$ and $\delta(A + f(B)/B + f(B)) = 1$ we obtain $\delta(\tilde{A}/B+f(B)) = 0$.

There is a sequence of minimal extensions $B+f(B) = C_0 \leq \cdots \leq C_n = \tilde{A}$. Moreover, as $\delta(\tilde{A}/A+f(B)) = 0$, all these extensions are either prealgebraic or algebraic. In particular, we can iteratively extend $f$ to each of them using Proposition \ref{der-ext-sol}, since $\delta(C_{i+1}/C_i)=0$ imposes at each step that $C_{i+1}+f(C_i)\leq V_\mu$.

\end{proof}

\begin{theorem}\label{der-extension}

Let $B \leq V_{\mu}$ finite and $f : B \rightarrow B + f(B)$ be a partially defined graded derivation, with $B + f(B) \leq V_{\mu}$. Let $a \in V_{\mu}$. There exists $B \leq A$ with $a \in A$ and a graded extension $f : A \rightarrow A + f(A)$ to $A$ with $A \leq V_{\mu}$ and $A + f(A) \leq V_{\mu}$.

\end{theorem}

\begin{proof}

Let $A'$ be the self-sufficient closure of $\sev{B,a}$. There is a sequence $B \leq A'_0 \leq \cdots \leq A'_n = A'$ of minimal extensions. Moreover, we know that $\delta(A'/B) \leq 1$, thus at most one of these extensions is transcendental. 

We extend iteratively $f$ to $A'_i$. While $A'_{i+1}+f(A'_i) \leq V_\mu$, we apply Proposition \ref{der-ext-sol}. If it is the case for all $i<n$, we obtain an extension $f$ to $A=A'\supseteq \sev{B,a}$.

Otherwise, there is $i_0<n$ such that $f$ is extended to $A'_{i_0}$ and $A'_{i_0+1}+f(A'_{i_0})$ is not self-sufficient in $V_{\mu}$. By Lemma \ref{der-sol-min}, $\delta(A'_{i_0+1}/A'_{i_0})=1$ and we can extend $f$ to $A_{i_0+1}$ where  $A_{i_0+1}$ is the self-sufficient closure of $A'_{i_0+1}+f(A'_{i_0})$. 

Now define $A_i=A'_i+ A_{i_0+1}$ for $i>i_0+1$. Since $\delta(A'/B) \leq 1$, we have $\delta(A'_{i+1}/A'_i)=0$ for all $i>i_0$. 

This implies that $A_i \leq V_{\mu}$ for all $i > i_0 $. Indeed, if $A_i \subset D$ for some $D$, then:

\begin{align*}
\delta(D/A_i) & = \delta(D) - \delta(A_{i_{0}+1} + A_i ') \\
& \geq \delta(D) - \delta(A_{i_{0}+1}) -\delta(A_i ') + \delta(A_{{i_0}+1} \cap A_i ') \text{ (submodularity)}\\
& = \delta(D/A_{i_0 + 1}) - \delta(A_i ' / A_{i_0 +1} \cap A_i ') \\
& \geq -\delta(A_i ' / A_{i_0 +1} \cap A_i ') \text{ as } A_{i_0 + 1} \leq V_{\mu}\\
& \geq 0 
\end{align*}

\noindent where the last line is a consequence of the fact that $A_i'$ is a strong extension of $A_{i_0 + 1}'$ with $\delta(A_i '/A_{i_0 +1})= 0$, and $A_{i_0 + 1}' \subset A_{i_0 +1} \cap A_i ' \subset A_i '$.

This implies that $\delta(A_{i+1}/A_i) = 0$ for all $i > i_0$, as:

\begin{align*}
\delta(A_{i+1}/A_i) & = \delta(A_i + A_{i+1}' / A_i) \\
& \leq \delta(A_{i+1}'/A_i \cap A_{i+1}') \text{ (submodularity)}\\
& \leq 0
\end{align*}

\noindent where the last line is obtained by similarly considering the strong extension $A_i ' \leq A_{i+1}'$ and $A_i ' \subset A_i \cap A_{i+1} ' \subset A_{i+1}'$. We get $\delta(A_{i+1}/A_i) = 0$ by strongness.

Again by Lemma \ref{der-sol-min}, for $i>i_0$, we have iteratively $A_{i+1}+f(A_i) \leq V_\mu$ and we can extend $f$ to $A_{i+1}$ by Proposition \ref{der-ext-sol}. Thus, we obtain an extension $f$ to $A=A_n \supseteq A' \supseteq \sev{B,a}$.
\end{proof}

By a direct induction, any \emph{finite} partially defined graded derivation can be extended to $M_\mu$:

\begin{cor}\label{der-b-a-f}
Any partially defined graded derivation $f : B \rightarrow B + f(B)$, with $B \leq B + f(B) \leq V_{\mu}$ and $B$ finite, can be extended to a graded derivation $f$ on $M_{\mu}$ (i.e. $f(V_\mu) \subseteq V_\mu$).

\end{cor}

\section{A Cover without the CBP}

We are now ready to define a cover of the sort $V_\mu$ in $M_{\mu}$ that will not have the canonical base property. It will closely mirrors the example constructed in \cite{hrushovski2013canonical}.

Let us first recall a few definitions, valid in any superstable theory. 

\begin{definition}

Let $\mathcal{P}$ be a family of partial types. A stationary type $p \in S(A)$ is $\mathcal{P}$-internal (resp. almost $\mathcal{P}$-internal) if there is a set of parameters $C$, such that for any realization $a \models p$ independent from $C$ over $A$, there is a tuple $e$ of realizations of types in $\mathcal{P}$, each based over $C$, such that $a \in \dcl(e,C)$ (resp $a \in \acl(e,C)$). 

\end{definition}

\begin{fact}\label{fun-sys}

If all partial types in $\mathcal{P}$ are over some parameters $A$ and $p \in S(A)$, the parameters $C$ can be picked as a Morley sequence $a_1, \cdots , a_n$ in $p$, called a fundamental system of solutions (see \cite{pillay1996geometric}, Chapter 7, Lemma 4.2).

\end{fact}

More generally, a stationary type $p \in S(A)$ is said to be $\mathcal{P}$-analysable if there is $a \models q$ and $a= a_n, a_{n-1}, \cdots , a_1$ such that $\tp(a_1/A)$ is almost $\mathcal{P}$-internal, and for all $n > i \geq 1$ we have that $\tp(a_{i+1}/A a_i)$ is stationary, $\mathcal{P}$-internal and $a_i \in \dcl(a_{i+1}A)$.

\begin{definition}

A theory has the canonical base property (CBP) if (possibly working over some parameters) for any tuples $a,b$, if $\tp(a)$ has finite Lascar rank and $b = \Cb(\stp(a/b))$, then $\stp(b/a)$ is almost $\mathcal{P}$-internal, where $\mathcal{P}$ is the family of Lascar rank one types.

\end{definition}

In this last section, we construct a structure without the CBP that is interpretable in $M_{\mu}$ with its 2-nilpotent graded Lie algebra structure. As a consequence, this structure will be CM-trivial:

\begin{definition}

A stable theory $T$ is CM-trivial if whenever $A \subset B$ are parameters and $c$ is a tuple satisfying $\acl^{\mathrm{eq}}(c,A) \cap \acl^{\mathrm{eq}}(B) = \acl^{\mathrm{eq}}(A)$, then $\Cb(\stp(c/A)) \subset \Cb(\stp(c/B))$.

\end{definition}

In particular, our structure will not interpret a field, by a result of Pillay \cite{pillay2000note}. Let us now define it.

Let $Q$ be a sort, given by $M_{\mu}$, and let $P= V_{\mu}$. The second sort $S$ is given (as a set) by $V_\mu^2$. 
We equip $Q = M_{\mu}$ with all its $L$-structure, i.e. its structure of 2-nilpotent graded Lie algebra. Moreover, we equip $S$ with:

\begin{itemize}
\item the projection $\pi : S \rightarrow P$, with $\pi((a,u)) = a$.
\item the group action $(P,+) \times S \rightarrow S$ given by $b * (a,u) = (a, u +b)$.
\item the group law $+ : S^2 \rightarrow S$, given by the addition on $V_\mu$.
\item for any $\emptyset$-definable set $W$ in $V_\mu^{2n}$ given by $\sum\limits_{i<n} [x_i, y_i]=0$, a relation $T_W$ on $S$ given by 
\[\sum\limits_{i<n} [x_i, y_i]=0 \text{ and } \sum\limits_{i<n} \big([u_i, y_i] + [x_i, v_i]\big)=0\]
for $(x_i,u_i)$ and $(y_i,v_i)$ in $S$.

\end{itemize}

We denote this two sorted structure $N$. Remark that the relations $T_W$ are, formally, defined exactly like tangent bundles in an algebraically closed field.

Before proving that $N$ does not have the canonical base property, let us explain why it is CM-trivial. First, the Lie algebra $M_{\mu}$ is CM-trivial, as proven by Baudish in his construction \cite{baudisch1996new}, and therefore so is $M_{\mu}^{\mathrm{eq}}$. Moreover, our structure $N$ is simply a reduct of $M_{\mu}^{\mathrm{eq}}$ (in fact, of $(M_\mu, V_\mu^2)$). By a result of N\"{u}bling \cite{nubling2005reducts}, any reduct of a finite Morley rank CM-trivial theory is also CM-trivial.

The CBP refers to (almost) internality to the family of non-locally modular strongly minimal types. Thus we need to identify this family in $N$. Note that by construction, this structure is 2-analyzable in $P = V_{\mu}$, which is strongly minimal and non-locally modular. Thus $N$ is $\aleph_1$-categorical, and $P$ is the only non-locally modular strongly minimal set, up to non-orthogonality, and we will identify (almost) internality to strongly minimal non-locally modular types and (almost) internality to $P$. 
To prove that $N$ does not have the CBP, the key observation is:

\begin{lemma}

For any graded derivation $f$ of $M_{\mu}$ (i.e. $f(V_{\mu}) \subseteq V_{\mu}$), the map:
\begin{align*}
\sigma_{f} : N & \rightarrow N \\
(a,u) \in S& \rightarrow (a,u+f(a)) \\
b \in Q & \rightarrow b
\end{align*}
\noindent is an automorphism of $N$, fixing $Q$ pointwise by definition.

\end{lemma}

\begin{proof}

Let $f$ be a graded derivation and $\sigma_f$ be the associated map. It is obviously a bijective map, so we only have to prove that it preserves the functions and relations of $N$. As it is the identity on  $Q$, it preserves its relations and functions. Preservation of the projection, group law and group action are immediate.

In order to show that the definable sets $T_W$ are preserved, consider a tuple $((a_1,u_1), \ldots , (a_n,u_n),(b_1,v_n), \ldots , (b_n,v_n))\in T_W$, with $W$ given by $\sum\limits_{i=1}^n [x_i, y_i] = 0$. We have:

\begin{align*}
\sum\limits_{i=1}^n \big([u_i + f(a_i), b_i] + [a_i,v_i + f(b_i)]\big) & = \sum\limits_{i=1}^n \big([u_i,b_i] + [a_i, v_i]\big) + \sum\limits_{i=1}^n f([a_i,b_i]) \\
& =  \sum\limits_{i=1}^n  \big([u_i,b_i] + [a_i, v_i]\big) + f\left(\sum\limits_{i=1}^n [a_i,  b_i]\right) \\
& = \sum\limits_{i=1}^n  \big([u_i,b_i] + [a_i, v_i]\big) \\
& = 0
\end{align*}

\noindent and thus \[\big(((a_1,u_1+f(a_1)), \ldots , (a_n,u_n+f(a_n)),(b_1,v_n+f(b_1)), \ldots , (b_n,v_n+f(b_n))\big) \in T_W.\]

\end{proof}

Our previous work allowed us to construct many derivations on $M_{\mu}$, which will thus give rise to automorphisms of $N$ fixing $Q$. The key consequence of this is:

\begin{cor}\label{not-internal}

For any $(a,u) \in S$, the type $\tp(a,u)$ is not almost $P$-internal.

\end{cor}

\begin{proof}

First note that the structure $N$ is a countable saturated model of its theory because $M_{\mu}$ is a countable saturated model of $T_{\mu}$. 
 
Thus, almost $P$-internality of a type over $\emptyset$ should be witnessed in $N$ since  we are only considering type-definable sets over finite sets of parameters. 

Namely, if $\tp(a,u)$ was almost $P$-internal, there would exist, in $N$, a sequence $(b_1,v_1), \cdots , (b_n,v_n)$ of realizations of $\tp(a,u)$, independent over $(a,u)$, and a tuple $\oc \in P$ such that $(a,u) \in \acl((b_1,v_1), \cdots , (b_n,v_n), \oc)$ (we are using Fact \ref{fun-sys} here). Let us suppose so, and try to derive a contradiction.

A consequence of independence is that $a, b_1, \cdots, b_n$ are independent, as points in $M_{\mu}$. Thus we have $\la a, b_1, \cdots ,b_n \ra \leq V_{\mu}$, and they do not satisfy any Lie algebra equations, meaning $N(\langle a, b_1, \cdots, b_n \rangle ) = \{ 0 \}$. Thus, if $e$ is any non-zero point of $V_{\mu}$, independent from $a,b_1, \cdots , b_n$, we have $\langle a, b_1, \cdots, b_n ,e \rangle \leq V_\mu$ and the linear map

\begin{align*}
f : \langle a, b_1, \cdots, b_n \rangle  & \rightarrow  \langle a, b_1, \cdots, b_n ,e \rangle\\
a & \rightarrow  e \\
b_i & \rightarrow  0
\end{align*}

\noindent gives a partially defined graded derivation. Applying Corollary \ref{der-b-a-f}, we extend $f$ to a graded derivation of $M_{\mu}$, which yields an automorphism $\sigma_{f}$ of $N$, fixing $Q$ pointwise. Moreover, this automorphism fixes $(b_i,v_i)$ for all $i$, and $\sigma_f ((a,u)) = (a,u+e)$. As we can find such an automorphism for any $e \in V_{\mu}$ independent from $a, b_1, \cdots ,b_n$, the orbit of $(a,u)$ under automorphisms fixing $(b_1,v_1), \cdots , (b_n,v_n), \oc$ is infinite, which is a contradiction.

\end{proof}

To prove that our structure does not have the CBP, we will use what is called the "group version" of the CBP, observed by Pillay and Ziegler in \cite{pillay2003jet}:

\begin{prop}\cite[Fact 1.3]{hrushovski2013canonical}\label{group-version-CBP}

Assume $T$ has the CBP. Let $G$ be a definable group, let $a \in G$, and assume that $p = \stp(a/A)$ has finite stabilizer. Then $p$ is almost internal to the family of non-locally modular strongly minimal types. 

\end{prop}

We can now prove:

\begin{theorem}

The structure $N$ does not have the CBP.

\end{theorem}

\begin{proof}

Using Proposition \ref{group-version-CBP}, we need to find a tuple $((a_1,u_1),\ldots,(a_n,u_n)) \in S^n$ such that $\tp(((a_1,u_1),\ldots,(a_n,u_n))$ has finite stabilizer for the group law of $S$. Indeed, such a type is never $P$-internal, by Corollary \ref{not-internal}.

To do so, we will, \emph{mutatis mutandis}, use the proof of Theorem 3.7 in \cite{hrushovski2013canonical}. For the convenience of the reader, we repeat the proof of Hrushovski, Palac\'{i}n and Pillay here.

Fix $(a,u),(b,v),(c,r),(d,s)$ generic points of $T_W$, where $W$ is the definable set given by $ \{(x,y,z,w) \in P^4, [x, z] + [y, w] = 0 \}$. Let $q = \tp((a,u),(b,v),(c,r),(d,s))$. Concretely, this means that we have the following equalities: 

\begin{align*}
[a , c] + [b , d] & = 0 \\
[u , c ]+ [a , r] + [v , d] + [b , s] & = 0
\end{align*}

We will prove that $\mathrm{Stab}(q)$ is trivial. First, we show that $\mathrm{Stab}(a,b,c,d)$ is trivial in $(V_{\mu},+)^4$. Suppose that $e_1,e_2,e_3,e_4 \in \mathrm{Stab}(a,b,c,d)$, and is independent from $a,b,c,d$. Then we have:

\begin{align*}
[a+e_1 , c+e_3] + [b + e_2 , d + e_4] = 0
\end{align*}

\noindent which simplifies into:

\begin{align*}
[a , e_3] + [e_1 , c] + [e_1 , e_3] + [b , e_4] + [e_2 , d] +[ e_2 , e_4] = 0
\end{align*}

\noindent which contradicts the independence assumption, unless $e_i = 0$ for $i = 1, \cdots ,4$. Thus $\mathrm{Stab}(a,b,c,d)$ is trivial.

Hence any element in the stabilizer of $q$ is of the form $(0,x),(0,y),(0,w),(0,z)$. Pick such a tuple, independent from $(a,u),(b,v),(c,r),(d,s)$.
Then we have:

\begin{align*}
[u+x, c] + [a , r+w] + [v+y , d] + [b , s+z]  = 0
\end{align*}

\noindent giving us:

\begin{align*}
[x , c] + [a , w] + [y , d] + [b , z] = 0.
\end{align*}

If $x,y,z,w$ were elements of $V_{\mu}$, independent over $a,b,c,d$, we could directly conclude that $x=y=z=w = 0$. However, this is not the case, as in the structure $N$, only the tuples $(0,x),(0,y),(0,z),(0,w)$ exist. We will now find elements $x'',y'',z'',w'' \in V_{\mu}$ satisfying the same equation and independence.

Let $(0,x'),(0,y'),(0,w'),(0,z')$ be another tuple in the stabilizer of $q$, this time independent from both $(a,u),(b,v),(c,r),(d,s)$ and $(0,x),(0,y),(0,w),(0,z)$. We similarly obtain:
\begin{align*}
[x' , c] + [a , w'] + [y' , d] + [b , z'] = 0.
\end{align*}

Pick $x'',y'',w'',z'' \in P$ such that $x'' * (0,x) = (0,x')$ (i.e. $x'' = x' - x$ in the full structure), and similarly for $y,w$ and $z$. Again we have:
\begin{align*}
[x'' , c] + [a , w''] + [y'' , d] + [b , z''] =0.
\end{align*}

A quick forking computation yields that $x'',y'',w'',z''$ are independent from $a,b,c,d$ in $Q = M_{\mu}$. This yields $x'' = y'' = w'' = z'' = 0$, and thus $(x,y,w,z) = (x',y',w',z')$. As $(0,x),(0,y),(0,w),(0,z)$ and $(0,x'),(0,y'),(0,w'),(0,z')$ are independent over the empty set, this implies that $\mathrm{Stab}(q)$ is trivial.

\end{proof}

\begin{remark}

It is unclear how many different, up to interpretability, new theories without the CBP this construction yields. Indeed, for any $q > 2$ and good map $\mu$, we obtain a structure $N = N_{q,\mu}$. There are countably many choices for $q$, and uncountably many for $\mu$, once $q$ is fixed. It remains to be seen if any $N_{q,\mu}$ can be interpreted in another. 

\end{remark}

\bibliography{biblio}
\bibliographystyle{plain}

\end{document}